\documentclass{article}
\usepackage{amsmath}
\usepackage{amsthm}
\usepackage{mathrsfs}
\usepackage{latexsym}
\usepackage{amssymb}
\usepackage{amscd}
\usepackage[dvips]{graphics}
\usepackage[all,cmtip]{xy}

\DeclareSymbolFont{AMSb}{U}{msb}{m}{n}
\DeclareMathSymbol{\N}{\mathbin}{AMSb}{"4E}
\DeclareMathSymbol{\Z}{\mathbin}{AMSb}{"5A}
\DeclareMathSymbol{\R}{\mathbin}{AMSb}{"52}
\DeclareMathSymbol{\Q}{\mathbin}{AMSb}{"51}
\DeclareMathSymbol{\I}{\mathbin}{AMSb}{"49}
\DeclareMathSymbol{\C}{\mathbin}{AMSb}{"43}

\newcommand{\dbl}{[\hspace{-0.2ex}[}
\newcommand{\dbr}{]\hspace{-0.2ex}]}
\newcommand{\db}[1]{\dbl {#1} \dbr}
\newcommand{\res}[1]{\hspace{-0.6mm}\downarrow_{\hspace{-0.25mm}{#1}}}
\newcommand{\ind}[1]{\hspace{-0.6mm}\uparrow^{\hspace{-0.25mm}{#1}}}
\newcommand{\ctens}{\widehat{\otimes}}

\newcommand{\iso}{\cong}
\newcommand{\invlim}{\underleftarrow{\textnormal{lim}}\,}
\newcommand{\ds}{\raisebox{0.5pt}{\,\big|\,}}
\newcommand{\onto}{\twoheadrightarrow}

\newcommand{\id}{\textnormal{id}}

\newcommand{\Hom}{\textnormal{Hom}}

\newcommand{\norm}[1]{\textnormal{N}_{#1}}
\newcommand{\End}{\textnormal{End}}

\newcommand{\rad}{\textnormal{rad}}

\numberwithin{equation}{section}

\title{Modular Representations Of Profinite Groups}
\author{John William MacQuarrie}

\begin{document}

\newtheorem{defn}[equation]{Def{i}nition}
\newtheorem{prop}[equation]{Proposition}
\newtheorem{lemma}[equation]{Lemma}
\newtheorem{theorem}[equation]{Theorem}
\newtheorem{corol}[equation]{Corollary}
\newtheorem{question}[equation]{Questions}

\maketitle

\section{Introduction}

The modular representation theory of a finite group $G$ attempts to
describe the modules over the group algebra $kG$, where $k$ is a
field of characteristic $p$ dividing the order of $G$. Under these
circumstances $kG$ is not semisimple and the vast majority of
$kG$-modules are not completely reducible.  Towards an understanding
of these modules the important concept of relative projectivity has
been considered in some depth.

Modular representation theory seems very well-suited for
consideration in the wider context of profinite groups.  If $G$ is a
profinite group and $k$ is a finite field, then there is a very
natural profinite analogue of the group algebra for $G$, and hence
of the corresponding profinite modules.  There is also a
well-defined Sylow theory of profinite groups that in particular
allows us to consider analogues for $p$-subgroups. The close
connection between a profinite object and its finite quotients
allows us to generalize several foundational results of modular
representation theory to a much wider universe of groups.

We give here an indication of our approach and the main results.  In Section \ref{relative projectivity section} we define the concept of relative projectivity for a $k\db{G}$-module and prove a characterization of finitely generated relatively $H$-projective modules, where $H$ is a closed subgroup of $G$ (Theorem \ref{relprojcharfinished}).  Of particular note in this characterization is the fact that a $k\db{G}$-module is relatively $H$-projective if and only if it is relatively $HN$-projective for every open normal subgroup $N$ of $G$.  In Section \ref{vertex section} we introduce the vertex of an indecomposable finitely generated module, proving existence (Corollary \ref{vertex exists}) and uniqueness up to conjugation in $G$ (Theorem \ref{verticesconjugate}).  Crucial in the proof of \ref{verticesconjugate}, and elsewhere, is the helpful fact that a finitely generated indecomposable $k\db{G}$-module has local endomorphism ring (Proposition \ref{fg module has LER}).  In Section \ref{sources section} we introduce the concept of source, but note that this object seems less natural in the profinite category than it does in the finite case.  We prove under additional hypotheses that finitely generated sources are unique up to conjugation (Theorem \ref{sources conjugate}).  In the last section we prove an analogue of Green's indecomposability theorem for modules over the completed group algebra of a virtually pro-$p$ group (Theorem \ref{profinite indecomposability}).  To do this, we first show that an important characterization of absolutely indecomposable modules, known to hold for finite groups, also holds for virtually pro-$p$ groups (Theorem \ref{profinite absindec iff End/rad is k}).  Finally, we answer the question of what happens when the module in question is not necessarily absolutely indecomposable, showing that the induced summands are isomorphic (Theorem \ref{profinite inducedsummandsisomorphic}).  In many proofs we utilize a class of quotient modules known as coinvariants.  These give a natural inverse system for a module with some very useful properties, many of which are elucidated in Section \ref{prelims}.

It is hoped that in the future, results in the area will have number theoretic applications (to Iwasawa algebras or to Galois theory, for instance) as well as being of interest from a purely algebraic perspective.

There are excellent books available covering the prerequisite
material of this paper.  For a detailed introduction to profinite
objects see \cite{wilson}, or for an explicitly functorial approach
well suited to our needs see \cite{ribzal}.  For the modular
representation theory of finite groups see \cite{alperin},
\cite{benson} or the encyclopedic \cite{curtisandreiner1}.  Our
discussion will for the most part follow the path laid out in the
seminal paper \cite{greenvertexpaper} of J.A. Green, published
almost exactly 50 years ago.

\section{Preliminaries}\label{prelims}

Throughout our discussion let $k$ be a finite field of
characteristic $p$ and let $G$ be a profinite group.  We define
well-known profinite analogues of the natural objects of modular
representation theory.  Denote by $k\db{G}$ the \emph{completed
group algebra} of $G$ - that is, the completion of the abstract
group algebra $kG$ with respect to the open normal subgroups of $G$.
Since $k$ is finite and $G$ is profinite, the completed group
algebra $k\db{G}$ is profinite.  A profinite $k\db{G}$-module is a
profinite additive abelian group $U$ together with a continuous map
$k\db{G}\times U\to U$ satisfying the usual module axioms.  It
follows from \cite[5.1.1]{ribzal} that $U$ is the inverse limit of
an inverse system of finite quotient modules of $U$.  If not
explicitly stated, our modules are \emph{profinite left} modules.

Let $H$ be a closed subgroup of $G$.  If $W$ is a right
$k\db{H}$-module and $V$ is a left $k\db{H}$-module, then we denote
by $W\ctens_{k\db{H}}V$ the \emph{completed tensor product} of $W$
and $V$ over $k\db{H}$ \cite[5.5]{ribzal}.  This is the natural
profinite analogue of the abstract tensor product and satisfies most
of the properties one would expect.  If either $W$ or $V$ is
finitely generated as a $k\db{H}$-module then the completed tensor
product and abstract tensor product coincide
\cite[5.5.3(d)]{ribzal}.  Now let $V$ be a profinite
$k\db{H}$-module and define the \emph{induced} $k\db{G}$-module
$V\ind{G}$ as $k\db{G}\ctens_{k\db{H}}V$ with action from $G$ on the
left factor.  If $U$ is a $k\db{G}$-module then the
\emph{restricted} $k\db{H}$-module $U\res{H}$ is the module $U$ with
coefficients restricted to $k\db{H}$.

A profinite $k\db{G}$-module $U$ is said to be \emph{finitely
generated} if there is a finite subset $\{u_1,\hdots, u_n\}$ of $U$
with every element of $U$ a $k\db{G}$-linear combination of the
elements $u_1,\hdots,u_n$.  Thus $U$ is the module abstractly
generated by the given finite subset, but by \cite[7.2.2]{wilson}
this module is in fact profinite.

Whenever $U,W$ are profinite $k\db{G}$-modules, denote by $\Hom_{k\db{G}}(U,W)$ the
$k$-module of continuous $k\db{G}$-module homomorphisms from $U$ to $W$.  We sketch
proofs for some properties of this object that do not seem to be
explicitly mentioned in the literature.

\begin{lemma} \label{Hom(U,W)isinverselimit}
Let $U$ and $W=\invlim_I W_i$ be profinite $k\db{G}$-modules.  then there is a
topological isomorphism
$$\Hom_{k\db{G}}(U,W)\iso \invlim_{i\in I}\Hom_{k\db{G}}(U,W_i),$$
where each set of maps is given the compact-open topology.
\end{lemma}

\begin{proof}
Abstractly this is essentially the definition of inverse limit.
Using basic properties of the compact-open topology it is easily
verified that the obvious isomorphism is a homeomorphism.
\end{proof}

\begin{corol}\label{fg homs profinite}
If $U$ is a finitely generated profinite $k\db{G}$-module and $W$ is
a profinite $k\db{G}$-module, then $\Hom_{k\db{G}}(U,W)$ is
profinite.
\end{corol}

If $H$ is a closed subgroup of $G$ ($H\leq_C G$) the functor
$(-)\ind{G}_H$ is left adjoint to $(-)\res{H}^G$.  The unit
$\eta:1\to (-)\ind{G}\res{H}$ is given by $\eta_V(v)=1\ctens v$ and
the counit $\varepsilon: (-)\res{H}\ind{G}\to 1$ by
$\varepsilon_U(g\ctens u)=gu$. In particular we have the following:

\begin{lemma}\label{H hom extends by induction}
Let $H\leq_C G$ and $V$ a $k\db{H}$-module.  Having identified $V$
with $1\ctens_{k\db{H}}V\subseteq V\ind{G}$, every continuous
$k\db{H}$-module homomorphism $V\to U\res{H}$ extends uniquely to a
continuous $k\db{G}$-module homomorphism $V\ind{G}\to U$.
\end{lemma}

The following result will also be of use.  For the definition of a filter base see \cite[1.2]{wilson}.

\begin{lemma}\label{ResInd inverse limit}
Let $U$ be a $k\db{G}$-module and let $\{W_i\,|\,i\in I\}$ be a
filter base of open subgroups of $G$.  Then $U\res{W}\ind{G}\iso
\invlim_i\, U\res{W_i}\ind{G}$, where $W=\bigcap W_i$.
\end{lemma}

\begin{proof}
This follows from \cite[5.2.2, 5.5.2, 5.8.1]{ribzal}.
\end{proof}

If $U$ is a finitely generated $k\db{G}$-module then we can give a
reasonably explicit inverse system for $U$ using \emph{coinvariant}
quotient modules.  If $N$ is a closed normal subgroup of $G$, then
the coinvariant module $U_N$ is defined as $k\ctens_{k\db{N}}U$,
where the left factor $k$ is the trivial $k\db{G}$-module.  The
action of $G$ on $U_N$ is given by $g(\lambda\ctens u)=\lambda\ctens
gu$. In tensor product notation we usually denote $U_N$ by
$k\ctens_N U$. The module $U_N$ can usefully be described as
follows:

\begin{lemma}\label{coinvUP}
$U_N$ together with the canonical projection map $\varphi_N:U\to
U_N$ is (up to isomorphism) the
unique $k\db{G}$-module on which $N$ acts trivially and satisfying
the following universal property:

Every continuous $k\db{G}$-module homomorphism $\rho$ from $U$ to a
profinite $k\db{G}$-module $X$ on which $N$ acts trivially factors
uniquely through $\varphi_N$.  That is, there is a unique continuous
homomorphism $\rho':U_N\to X$ such that $\rho'\varphi_N=\rho$.
\end{lemma}

Note that $N$ is in the kernel of the action of $k\db{G}$ on $U_N$,
so that $U_N$ can be considered as a $k\db{G/N}$-module.  It follows
that if $N$ is open and $U$ is finitely generated then $U_N$ is
finite. From properties of the completed tensor product (which in
this case is the same as the abstract tensor product) it is also
easy to check that the operation $(-)_N$ is a right exact functor
from the category of (finitely generated) $k\db{G}$-modules to the
category of (finitely generated) $k\db{G/N}$-modules.

We collect here several properties of coinvariant modules. First a
list of important technical details:

\begin{lemma}\label{coinvariant technicals}
Let $G$ be a profinite group, $N,M$ closed normal subgroups of $G$
with $N\leq M$ and $H$ a closed subgroup of $G$.  Let $U,W$ be
$k\db{G}$-modules and let $V$ be a $k\db{H}$-module.  Then
\begin{enumerate}
\item $(U_N)_M$ is naturally isomorphic to $U_M$.
\item $(U\oplus W)_N\iso U_N\oplus W_N$.
\item $V_{H\cap N}$ is naturally a $k\db{HN/N}$-module.
\item $(V\ind{G})_N\iso V_{H\cap N}\ind{G/N}$.
\item $U_N\res{HN/N}\iso (U\res{HN})_N$.
\end{enumerate}
\end{lemma}

\begin{proof}
The maps required for \emph{1.} are obtained by repeated use of the
universal property \ref{coinvUP}.  The remaining isomorphisms are
easily verified.
\end{proof}

We are primarily interested in coinvariant modules for the following
reason:

\begin{prop}\label{module invlimit of coinvariants}
If $U$ is a profinite $k\db{G}$-module, then $\{U_N\,|\, N\lhd_O
G\}$ together with the set of canonical quotient maps forms a
surjective inverse system with inverse limit $U$.
\end{prop}

\begin{proof}
It is clear that the maps $\varphi_{MN}:U_N\to U_M$ given by
$1\ctens_{N}u\mapsto 1\ctens_{M}u$ whenever $N\leq M$ are well
defined and give an inverse system of the $k\db{G}$-modules $U_N$.
It is also clear that we have a compatible set of maps
$\varphi_N:U\to U_N$ given by $u\mapsto 1\ctens_{N}u$.  We need only
show that $U$ is in fact the inverse limit. The maps $\varphi_N$ are the
components of a surjective map of inverse systems, giving a
continuous surjection $u\mapsto(1\ctens_{N}u)$ onto the limit by
\cite[1.1.5]{ribzal}, so we need only check that this map is
injective.

To do this we use the universal property \ref{coinvUP}.  By
definition $U$ is profinite, so is the inverse limit of some inverse
system of finite quotient modules.  Fix $u\neq 0$ in $U$ and some
finite quotient module $U/W$ in which the image of $u$ is non-zero.
Then since $U/W$ is finite some $N\lhd_O G$ must act trivially on
$U/W$, so that the quotient map $U\onto U/W$ factors through $U_N$
via $\varphi_N$. But if the image of $u$ under the composition is
non-zero then certainly the image of $u$ under $\varphi_N$ is
non-zero, and so the image of $u$ in $\invlim U_N$ is non-zero.
Thus, our map is injective and $U\iso \invlim_N U_N$, as required.
\end{proof}

\begin{lemma}\label{coinv non vanishing}
Let $G$ be a profinite group and $U$ a non-zero profinite
$k\db{G}$-module. Let $N$ be a closed pro-$p$ subgroup of $G$.  Then
$U_N\neq 0$.
\end{lemma}

\begin{proof}
First suppose that $G$ is a pro-$p$ group.  Since $U$ is profinite
it has a proper open submodule of finite index and hence a maximal
submodule $U'$, so the module $U/U'$ is simple.  But $U/U'$ is
finite, so can be regarded as a module for the finite $p$-group
$G/N_0$ for some $N_0\lhd_O G$.  The only simple module over a
finite $p$-group is $k$, so that $U/U'\iso k$ and we have a
surjection $\beta:U\onto k$. But every $N\lhd_C G$ acts trivially on
$k$, so that $\beta$ factors through every $U_N$, and thus $U_N\neq
0$ for each $N\lhd_C G$.

Now let $G$ be a general profinite group.  Since $U\res{N}$ is a
non-zero module for the pro-$p$ subgroup $N$, by the previous
paragraph $0\neq (U\res{N})_N\iso U_N\res{N}$.  Hence $U_N\neq 0$.
\end{proof}

The above result is particularly useful when $G$ is a virtually
pro-$p$ group so that $G$ has a basis of open normal pro-$p$
subgroups.  In this case the following result will be of the utmost
importance:

\begin{prop}\label{U_N indec}
Let $G$ be a virtually pro-$p$ group and let $U$ be an
indecomposable finitely generated $k\db{G}$-module.  Then there
exists some $N_0\lhd_O G$ such that $U_N$ is indecomposable for
every $N\leq N_0$.
\end{prop}

\begin{proof}
We work within the cofinal (see \cite[1.1.9]{ribzal}) inverse system
$\{U_N, \varphi_{MN}\}$ of coinvariant modules for which $N$ is a
pro-$p$ group.  Since $\varphi_{MN}$ is functorial and summands
cannot have zero image by \ref{coinv non vanishing}, we see that as
we move up our system the number of direct summands of the $U_N$
cannot increase.  It follows that for some $N_0\lhd_O G$ and any
$N\leq N_0$ the number $n$ of indecomposable summands of $U_N$ is
equal to the number of indecomposable summands of $U_{N_0}$. We take
the cofinal inverse system of those $N$ contained in $N_0$

For each $N$, let $s_N$ be a set $\{X_{N,1}, \hdots, X_{N,n}\}$ of
$n$ indecomposable submodules of $U_N$ intersecting pairwise in $0$
and having the property that $U_N$ is equal to the (direct) sum
$X_{N,1}\oplus\hdots\oplus X_{N,n}$.  Denote by $S_N$ the set of all
possible $s_N$ - a non-empty finite set.  We form a new inverse
system of the finite sets $S_N$ via the maps $\psi_{MN}:S_N\to S_M$
given by
$$\psi_{MN}(s_N)=\psi_{MN}(\{X_{N,1}, \hdots, X_{N,n}\})=\{\varphi_{MN}(X_{N,1}), \hdots, \varphi_{MN}(X_{N,n})\}.$$

Since each $S_N$ is finite and non-empty the inverse limit of this
system is non-empty by \cite[1.1.4]{ribzal}.  We fix once and for
all some element $(s_N)$ of $\invlim S_N$, and for each $N$ we
choose our direct sum decomposition of $U_N$ to be the one given to
us by $s_N$.

Recall that we are only considering those $N\lhd_O G$ contained in
$N_0$, so that each $U_N$ maps onto $U_{N_0}$.  Fix some
indecomposable summand $X_{N_0}$ of $U_{N_0}$ (an element of
$s_{N_0}$) and for each $N$ in our system define $X_N$ to be the
unique element of $s_N$ with $\varphi_{N_0N}(X_N)=X_{N_0}$. It is
now easy to check that $\{X_N,\varphi_{MN}|_{X_N}\}$ is an inverse
system of submodules of the $U_N$.  Denote the inverse limit of this
system by $X$ - a submodule of $U$.

We want to show that $X$ is a summand of $U$.  For each $N$, we have
a canonical inclusion map $X_N\hookrightarrow U_N$, and these maps
give a map of inverse systems $\{X_N\}\to\{U_N\}$ in which each
component splits.  This corresponds to an injection
$\iota:X\hookrightarrow U$.  For each $N$, let $P_N$ denote the
finite, non-empty set of projection maps $U_N\onto X_N$ splitting
the corresponding component of $\iota$.  The functoriality of
$(-)_N$ gives us an inverse system of the $P_N$, and an element of
the limit is a map of inverse systems corresponding to a splitting
$\pi:U\to X$ of $\iota$.  Thus $X$ is a direct summand of $U$. But
$X\neq 0$ since $X_N\neq 0$ and the maps $\varphi_{N}|_{X}$ are
surjective, so that since $U$ is indecomposable we must have $X=U$.
But now $X_N=\varphi_{N}(X)=\varphi_{N}(U)=U_N$ for each $N$, and
thus each $U_N$ is indecomposable, as required.
\end{proof}

\section{Relative Projectivity}\label{relative projectivity section}

Our main definition is completely analogous to the equivalent
definition for finite groups.

\begin{defn}\label{relprojdefn}
Let $G$ be a profinite group and let $H\leq_C G$.  Then a profinite
$k\db{G}$-module $U$ is \emph{relatively $H$-projective} if whenever
we are given a diagram
$$
\xymatrix{
    & U \ar[d]^{\varphi}\\
V \ar@{->>}[r]_{\beta} & W}
$$
of profinite $k\db{G}$-modules and continuous $k\db{G}$-module
homomorphisms, then there exists a continuous $k\db{G}$-module
homomorphism $\rho:U\to V$ with $\beta\rho=\varphi$ provided there
is a $k\db{H}$-module homomorphism with this property.
\end{defn}

As in the finite case, a projective module is precisely a
$1$-projective module in the definition above.  Our goal for this
section is to obtain a characterization of relatively $H$-projective
finitely generated $k\db{G}$-modules analogous to D.G. Higman's
characterization in the finite case, for which see
\cite[3.6.4]{benson}. We will also demonstrate two new
characterizations that are trivial in the finite case but of great
use in our more general setting.

\begin{lemma}\label{relprojchar1}
Let $G$ be a profinite group and $H$ a closed subgroup of $G$. If
$U$ is a profinite $k\db{G}$-module then the following are
equivalent:
\begin{enumerate}
\item $U$ is relatively $H$-projective.

\item If ever a continuous $k\db{G}$-epimorphism $V\onto U$ splits as a
$k\db{H}$-module homomorphism, then it splits as a $k\db{G}$-module
homomorphism.

\item $U$ is a direct summand of $U\res{H}\ind{G}$.

\item $U$ is a direct summand of a module induced from some
profinite $k\db{H}$-module.
\end{enumerate}
\end{lemma}

\begin{proof}
This is proved just as for finite groups so the details are omitted.
At several points we require \ref{H hom
extends by induction}.
\end{proof}

We give now two very useful characterizations of finitely generated
profinite $k\db{G}$-modules.  As is standard, we write $U\ds W$ to
mean that the profinite module $U$ is isomorphic to a direct summand
of the profinite module $W$ - of course we insist that the splitting
maps are continuous.

\begin{prop} \label{relproj-H-iff-HN}
Let $U$ be a finitely generated profinite $k\dbl G\dbr$-module, and
$H\leq_C G$.  Then $U$ is relatively $H$-projective if and only if
$U$ is relatively $HN$-projective for every $N\lhd_O G$.
\end{prop}

\begin{proof}
The `only if' statement is clear.  We need only show that if $U$ is relatively $HN$-projective for each $N$, then $U$ is relatively $H$-projective.

By \ref{ResInd inverse limit} we have
$U\res{H}\ind{G}\iso\invlim_{N\lhd_O G}
\{U\res{HN}\ind{G},\psi_{MN}\}$. We will form the required splitting
homomorphisms as limits of maps of inverse systems.

For each $N\lhd_O G$ the identity map $U\res{HN}\to U\res{HN}$
extends uniquely to a surjection $\pi_N:U\res{HN}\ind{G}\onto U$ by
\ref{H hom extends by induction}. By checking commutativity of the
relevant diagrams on $U\res{HN}$ it follows that
$\{\pi_N\,|\,N\lhd_O G\}$ is a surjective map of inverse systems.
This map yields a continuous surjective homomorphism
$\pi:U\res{H}\ind{G}\to U$.

We note that the map $U\to U\res{HN}\ind{G}$ given by $u\mapsto
1\ctens u$ is a $k\db{HN}$-homomorphism, and that it splits $\pi_N$.
Hence, since $U$ is $HN$-projective, we have that $\pi_N$ splits as
a $k\db{G}$-homomorphism.  Let $I_N$ denote the non-empty set of
$k\db{G}$-splittings of the map $\pi_N$.

Since $U$ is finitely generated, we have that
$\Hom_{k\db{G}}(U,U\res{HN}\ind{G})$ is compact by \ref{fg homs
profinite}.  Since the map from
$\textnormal{Hom}_{k\db{G}}(U,U\res{HN}\ind{G})$ to
$\textnormal{End}_{k\db{G}}(U)$ given by $\alpha\mapsto\pi_N\alpha$
is continuous, the inverse image of $\textnormal{id}_U$, which is
$I_N$, is closed and hence compact.

The maps $I_N\to I_M$ given by $\iota_N\mapsto \psi_{MN}\iota_N$
whenever $N\leq M$ make the $I_N$ into an inverse system of
non-empty compact sets, and this system has a non-empty inverse
limit by \cite[1.1.4]{ribzal}. By definition an element of this
limit is a compatible map of inverse systems $\{\iota_N\}:U\to
\{U\res{HN}\ind{G}\}$.  This map of systems yields a unique
$k\db{G}$-homomorphism $\iota:U\to U\res{H}\ind{G}$.

Now by the functoriality of $\invlim$ we have
$$\pi\iota=\invlim\pi_N\invlim\iota_N=\invlim\pi_N\iota_N=\invlim
\id_{U}=\id_U$$ so that $U\ds U\res{H}\ind{G}$, as required.
\end{proof}

We can refine this further into a condition relying only on the
finite quotients $U_N$.  The following lemma will help us here and
elsewhere:

\begin{lemma}\label{fg summands of coinvs are summands}
Let $U,W$ be finitely generated profinite $k\db{G}$-modules and let
$\mathcal{N}$ be a cofinal inverse system of open normal subgroups
of $G$.
\begin{itemize}
\item If $U_N\ds W_N$ for each $N\in \mathcal{N}$, then $U\ds W$.
\item If $U_N\iso W_N$ for each $N\in \mathcal{N}$, then $U\iso W$.
\end{itemize}
\end{lemma}

\begin{proof}
For each $N\in \mathcal{N}$, let $P_N$ denote the non-empty finite
set of surjections $\pi_N:W_N\onto U_N$ that split.  Whenever $N\leq
M$ define $\gamma_{MN}:P_N\to P_M$ by $\pi_N\mapsto (\pi_N)_M$. This
gives an inverse system of finite non-empty sets. Thus we have a
non-empty inverse limit, and we fix an element $(\pi_N)$ of this
limit.

For each $N$ we have a non-empty finite set $I_N$ of injections
$\iota_N:U_N\to W_N$ splitting $\pi_N$.  As above we have a map
$I_N\to I_M$ since
$$\pi_M(\iota_N)_M=(\pi_N)_M(\iota_N)_M=(\pi_N\iota_N)_M=\id_{U_M}$$
and again we have an inverse system.  An element $(\iota_N)$ of the
limit of this system is a splitting of $(\pi_N)$, and it follows
that $U\ds W$.

The second claim follows from the first by noting (for instance)
that if each map $\pi_N$ is injective, then so is the limit map
$\pi$.
\end{proof}

\begin{prop}\label{Urelproj iff UNrelproj}
Let $U$ be a finitely generated profinite $k\db{G}$-module, and
$H\leq_C G$.  Then $U$ is relatively $H$-projective if and only if
$U_N$ is relatively $HN$-projective for every $N\lhd_O G$.
\end{prop}

\begin{proof}
Fix $N\lhd_O G$.  If $U$ is $H$-projective then $U$ is
$HN$-projective.  Now the functoriality of $(-)_N$ ensures that
$$
U  \ds U\res{HN}\ind{G} \implies  U_N \ds (U\res{HN}\ind{G})_N
\implies  U_N \ds U_N\res{HN}\ind{G}
$$
so that $U_N$ is $HN$-projective.

To show the converse, fix some $M\lhd_O G$. We take the cofinal inverse system of $U_N$
for $N\lhd_O G$ and $N\leq M$, noting that each $U_N$ is relatively
$HM$-projective. We will show that $\invlim U_N=U$ is relatively
$HM$-projective.

By assumption we have $U_N\ds U_N\res{HM/N}\ind{G/N}$ for each $N$
in our inverse system.  But $U_N\res{HM/N}\ind{G/N}\iso
(U\res{HM}\ind{G})_N$ by \ref{coinvariant technicals}, so that for
each $N$ we have $$U_N\ds (U\res{HM}\ind{G})_N$$  and the claim now
follows from \ref{fg summands of coinvs are summands}.  Thus $U$ is
$HM$-projective for each $M$, and the result follows from
\ref{relproj-H-iff-HN}.
\end{proof}

\begin{defn}
If $H\leq_O G$ and $U, W$ are $k\db{G}$-modules, then the
\emph{trace map} $$\textnormal{Tr}_{H,G}:
\textnormal{Hom}_{k\db{H}}(U\res{H},W\res{H})\to
\textnormal{Hom}_{k\db{G}}(U,W)$$ is defined by $$\alpha\mapsto
\sum_{s\in G/H}s\alpha s^{-1}.$$
\end{defn}

For open $H$ the properties of the trace map given in
\cite[3.6.3]{benson} carry through just as for finite groups.  We
now complete our characterization of finitely generated relatively
$H$-projective $k\db{G}$-modules:

\begin{theorem}\label{relprojcharfinished}
Let $G$ be a profinite group, let $H\leq_C G$, and let $U$ be a
finitely generated profinite $k\db{G}$-module.  Then the following
are equivalent:
\begin{enumerate}
\item $U$ is relatively $H$-projective.

\item If ever a continuous $k\db{G}$-epimorphism $V\onto U$ splits as a
$k\db{H}$-module homomorphism, then it splits as a $k\db{G}$-module
homomorphism.

\item $U$ is a direct summand of $U\res{H}\ind{G}$.

\item $U$ is relatively $HN$-projective for every $N\lhd_O G$.

\item $U_N$ is relatively $HN$-projective for every $N\lhd_O G$.

\item $U$ is a direct summand of a module induced from some
profinite $k\db{H}$-module.

\item For every $N\lhd_O G$ there exists a continuous $k\db{HN}$-endomorphism $\alpha_N$ of $U$ such that
\mbox{$\textnormal{id}_U=\textnormal{Tr}_{HN,G}(\alpha_N)$}.
\end{enumerate}
\end{theorem}

\begin{proof}
The equivalence of statements 1,2,3,4,5 and 6 follows from results above.  That 6 implies 7 is shown as with the finite
proof \cite[3.6.4]{benson} after using the transitivity property
$X\ind{G}\iso X\ind{HN}\ind{G}$, where $X$ is the $k\db{H}$-module
coming from 6.  The proof that 7 implies 4 also mimics the finite case
\cite[3.6.4]{benson}.
\end{proof}

\section{Vertices}\label{vertex section}

Our definition for vertex is again in direct analogy with the
corresponding definition when the group $G$ is finite:

\begin{defn}
Let $U$ be a finitely generated indecomposable profinite
$k\db{G}$-module. A vertex $Q$ of $U$ is a closed subgroup of $G$
with respect to which $U$ is relatively projective, but such that
$U$ is not projective relative to any proper closed subgroup of $Q$.
\end{defn}

Unlike in the finite case, we must check that a vertex of $U$
exists.  We do this using the following lemma, which is useful in
other situations:

\begin{lemma}\label{projective relative to intersection}
Let $G$ be a profinite group and let $\mathcal{W}=\{W_i\,|\, i\in
I\}$ be a filter base of closed subgroups of $G$.  Let $U$ be a
finitely generated profinite $k\db{G}$-module that is projective
relative to each of the $W_i$. Then $U$ is projective relative to
$W=\bigcap_{i\in I}W_i$.
\end{lemma}

\begin{proof}
By \ref{relproj-H-iff-HN} it suffices to show that $U$ is relatively
$WN$-projective for arbitrary $N\lhd_O G$, so fix some such $N$.
From \cite[0.3.1(h)]{wilson} we have
$$WN=(\bigcap W_i)N=\bigcap W_iN.$$
The set $\{W_iN\,|\, i\in I\}$ is finite and thus for some
$W_{i1},\hdots,W_{in}\in \mathcal{W}$ we have
$$WN=W_{i1}N\cap\hdots\cap W_{in}N=(W_{i1}\cap\hdots\cap W_{in})N.$$
But now by hypothesis there is some $W_j\in \mathcal{W}$ with
$W_j\subseteq W_{i1}\cap\hdots\cap W_{in}$ so that $WN=W_jN$ for
some $j\in I$.  The result follows.
\end{proof}

\begin{corol}\label{vertex exists}
If $U$ is an indecomposable finitely generated profinite
$k\db{G}$-module, then a vertex of $U$ exists.
\end{corol}

\begin{proof}
Demonstrating the existence of a vertex amounts to showing that the
set $\mathcal{I}$ of closed subgroups of $G$ with respect to which
$U$ is relatively projective has a minimal element.

The set $\mathcal{I}$ is a partially ordered set when ordered by inclusion.  We need only show that any chain $\mathcal{J}$ in
$\mathcal{I}$ has a lower bound in $\mathcal{I}$, and then Zorn's
lemma gives us that $\mathcal{I}$ has a minimal element $Q$.  But
from \ref{projective relative to intersection} it follows that $U$
is projective relative to $R=\bigcap\{H\,|\, H\in \mathcal{J}\}$.
Thus $R$ is a lower bound for $\mathcal{J}$ and the result follows.
\end{proof}

Our main result for this section is that two vertices of a finitely
generated indecomposable $k\db{G}$-module $U$ are conjugate by an
element of $G$.  To prove this we require that $U$ have local
endomorphism ring. This is known when $G$ is virtually pro-$p$
\cite[2.1]{symondspermcom2} but by observing that profinite modules
are pure injective, we easily obtain the result for general $G$:

\begin{prop}\label{fg module has LER}
Let $G$ be a profinite group and let $U$ be an indecomposable
finitely generated $k\db{G}$-module.  Then $U$ has local
endomorphism ring.
\end{prop}

\begin{proof}
Let $E=\End_{k\db{G}}(U)$ be the ring of continuous
$k\db{G}$-endomorphisms of $U$, and note that by
\cite[7.2.2]{wilson} this ring coincides with the ring of abstract
$k\db{G}$-endomorphisms of $U$.   If $W$ were an abstract summand of
$U$ then $W$ would be finitely generated and hence profinite.  It
follows that $U$ is indecomposable as an abstract $k\db{G}$-module.
A profinite module is compact in the sense of \cite{warfield} and so
it follows from \cite[Theorem 2]{warfield} that $U$ is
pure-injective.

Now \cite[2.27]{facchini} tells us that the abstract endomorphism
ring of an abstract indecomposable pure-injective module is a local
ring. In particular, $E$ is a local ring.
\end{proof}

The relevance of this proposition is the following well-known general result.  We include a short proof for the reader's convenience.

\begin{lemma}\label{LER implies U summand of V or W}
Let $R$ be a ring with 1 and let $U,V,W$ be $R$-modules, where $U$ has local endomorphism ring.  If $U\ds (V\oplus W)$, then $U\ds V$ or $U\ds W$.
\end{lemma}

\begin{proof}
Whenever $X$ is isomorphic to a summand of $V\oplus W$, let $\pi_X, \iota_X$ denote splitting maps in the obvious way.  We have
$$\id_U=\pi_U(\iota_V\pi_V+\iota_W\pi_W)\iota_U=\pi_U\iota_V\pi_V\iota_U + \pi_U\iota_W\pi_W\iota_U$$
and since $U$ has local endomorphism ring (so in particular the non-units form an additive group), one of the summands on the right hand side (the first, say) is invertible.  Thus $\id_U=\pi_U\iota_V\pi_V\iota_U\gamma$ for some $\gamma\in \End(U)$, and now parenthesizing as $\id_U=(\pi_U\iota_V)(\pi_V\iota_U\gamma)$ demonstrates that $U\ds V$.
\end{proof}

\begin{theorem} \label{verticesconjugate}
Let $G$ be a profinite group, $U$ an indecomposable finitely
generated $k\db{G}$-module, and let $Q,R$ be vertices of $U$.  Then
there exists $x\in G$ such that $Q=xRx^{-1}$.
\end{theorem}

\begin{proof}
The module $U$ is relatively $R$-projective so is relatively
$RN$-projective for every open normal subgroup $N$ of $G$.  Fix some
such $N$.   Since $U\ds U\res{RN}\ind{G}$ and $U\ds U\res{Q}\ind{G}$
we have that $U$ is a direct summand of
$$U\res{RN}\ind{G}\res{Q}\ind{G} \iso\bigoplus_{s\in Q\backslash G/RN}
s(U\res{RN})\res{Q\cap sRNs^{-1}}\ind{G}$$ where the above sum
(coming from the Mackey decomposition formula
\cite[2.2]{symondsdoublecoset}) make sense since $RN$ is open so the
set of double coset representatives is finite.  But now since $U$
has local endomorphism ring, \ref{LER implies U summand of V or W} shows that
$$U\ds s(U\res{RN})\res{Q\cap sRNs^{-1}}\ind{G}$$
for some $s\in G$.  Thus $U$ is relatively $Q\cap
sRNs^{-1}$-projective.  But $Q$ is minimal, so we must have
$Q\subseteq sRNs^{-1}$.

Denote by $C_N$ the set of all $s\in G$ such that $Q\subseteq
sRNs^{-1}$.  Since $C_N$ is a union of sets of the form $QgRN$ for
appropriate $g\in G$, it follows that each $C_N$ is closed in $G$.
We thus have a collection of closed, non-empty sets
$\{C_N\,|\,N\lhd_O G\}$ and we wish to show that their intersection
is non-empty.  Let $N_1,\hdots, N_n$ be open normal subgroups of
$G$. Then $N_1\cap\hdots\cap N_n \lhd_O G$ and so by the previous
argument $C_{N_1\cap\hdots\cap N_n}\neq \emptyset$. This means that
there exists $s\in G$ such that $Q\subseteq sR(N_1\cap\hdots\cap
N_n)s^{-1}$ so that certainly for each $i\in\{1,2,\hdots, n\}$ we
have $Q\subseteq sRN_i s^{-1}$.  So $C_{N_1\cap\hdots\cap
N_n}\subseteq C_{N_1}\cap\hdots\cap C_{N_n}$ and thus
$C_{N_1}\cap\hdots\cap C_{N_n}\not=\emptyset$.  By compactness we
now have $\bigcap_N C_N$ is non-empty.

It follows that there is some $x\in G$ such that
\begin{align*}
Q &\subseteq xRNx^{-1}     &\forall N\lhd_O G     \\
Q &\subseteq \bigcap\{xRx^{-1}N\,|\, N\lhd_O G\}     \\
Q &\subseteq xRx^{-1} & \hbox{by }\cite[0.3.3]{wilson}.
\end{align*}
Repeating the same argument with $Q$ and $R$ interchanged, we find
$y\in G$ such that $R\subseteq yQy^{-1}$.

But now $Q\subseteq xRx^{-1}\subseteq (xy)Q(xy)^{-1}$.  Since
profinite groups are well behaved under conjugation it follows that
$Q=(xy)Q(xy)^{-1}$, and so $Q=xRx^{-1}$ as required.
\end{proof}

For the background Sylow theory we require for the following results
see \cite[Chapter 2]{wilson}.

\begin{prop}\label{general Sylow-proj}
If $H$ is a closed subgroup of a profinite group $G$ containing a
$p$-Sylow subgroup of $G$, then any finitely generated profinite
$k\db{G}$-module $U$ is relatively $H$-projective.
\end{prop}

\begin{proof}
Since $U$ is finitely generated, by \ref{relproj-H-iff-HN} we need
only show that $U$ is relatively $HN$-projective for any given
$N\lhd_O G$.  Suppose we have a diagram as in \ref{relprojdefn} and
a continuous $k\db{HN}$-module homomorphism $\rho':U\to V$ making
the diagram commute. Since the supernatural number $|G:H|$ is
coprime to $p$, the finite number $|G:HN|$ is non-zero in the field
$k$.  Hence the continuous map
$$\rho=1/|G:HN|\sum_{s\in G/HN}s\rho' s^{-1}$$
is well defined, and as in the finite case we check that $\rho$ is a
$k\db{G}$-module homomorphism such that $\beta\rho=\varphi$.
\end{proof}

\begin{corol}
If $U$ is a finitely generated indecomposable $k\db{G}$-module, then
any vertex of $U$ is a pro-$p$ group.
\end{corol}

\begin{proof}
By \ref{general Sylow-proj}, $U$ has a pro-$p$ vertex, and now since
the set of pro-$p$ subgroups of $G$ is closed under conjugation the
result follows from \ref{verticesconjugate}.
\end{proof}

\section{Sources}\label{sources section}

For an indecomposable finitely generated module $U$ over a finite
group, there is attached to any vertex $Q$ of $U$ a finitely
generated indecomposable $kQ$-module $S$ with the property that
$U\ds S\ind{G}$.  This object is easily seen to be unique up to
conjugation by elements of $\norm{G}(Q)$.  If $G$ is a profinite or
even a pro-$p$ group, the corresponding notion of source seems less
natural, and even existence is not clear in general.  None-the-less,
we prove that if $G$ is virtually pro-$p$ and $U$ is an
indecomposable finitely generated $k\db{G}$-module with vertex $Q$
and finitely generated sources $S$ and $T$, then $S$ and $T$ are
conjugate in $\norm{G}(Q)$.

The following simple lemma will prove key:

\begin{lemma}\label{cofinal induced indecomposable}
Let $G$ be a virtually pro-$p$ group, let $H$ be a closed subgroup
of $G$ and let $V$ be a finitely generated indecomposable
$k\db{H}$-module. Then there exists a cofinal inverse system of
$N\lhd_O G$ for which each $V\ind{HN}$ is indecomposable.
\end{lemma}

\begin{proof}
For any $N\lhd_O G$ we have by \ref{coinvariant technicals} that
$$(V\ind{HN})_N\iso V_{H\cap N}\ind{HN/N}\iso V_{H\cap N}.$$
Since $V\iso \invlim_{N\lhd_O G}V_{H\cap N}$ it follows by \ref{U_N
indec} that there is a cofinal inverse system of $N\lhd_O G$ for
which $V_{H\cap N}$ and thus  $(V\ind{HN})_N$ is indecomposable. Now
since we can choose our system of $N$ to be pro-$p$ we have by
\ref{coinv non vanishing} that no non-zero summands of
$V\ind{HN}$ can become zero on taking coinvariants, and so
$V\ind{HN}$ is indecomposable.
\end{proof}

Recall that if $V$ is a profinite $k\db{H}$-module for $H\leq_C G$
and $x\in G$ then we denote by $x(V)$ the $k\db{xHx^{-1}}$-module
$x\ctens_{k\db{H}}V$ with action from $xHx^{-1}$ given by
$$xhx^{-1}(x\ctens v)=x\ctens hv.$$
The functor $x(-)$ is exact.  We include two technical facts about
how conjugation interacts with induction and coinvariants:

\begin{lemma}\label{comm commutes with ind 2}
Let $Q\leq_C H\leq_C G$, let $T$ be a $k\db{Q}$-module, and let
$x\in G$.  Then
$$x(T)\ind{xHx^{-1}}\iso x(T\ind{H}).$$
\end{lemma}

\begin{proof}
This is easily checked.
\end{proof}

\begin{lemma}\label{comm commutes with coinv}
Let $H\leq_C G$, $N\lhd_O G$, and let $T$ be a $k\db{H}$-module.
Then
$$(x(T))_{xHx^{-1}\cap N}\iso x(T_{H\cap N}).$$
\end{lemma}

\begin{proof}
If $K$ is the kernel of the canonical map $T\onto T_{H\cap N}$ then
the result follows by conjugating the exact sequence $K\to T\to
T_{H\cap N}$ by $x$.
\end{proof}

\begin{defn}
Let $G$ be a profinite group and let $U$ be a finitely generated
indecomposable profinite $k\db{G}$-module with vertex $Q$. A
\emph{source} of $U$ is an indecomposable $k\db{Q}$-module $S$ such
that $U\ds S\ind{G}$.
\end{defn}

If $G$ is a virtually pro-$p$ group then our primary unanswered
question is whether a finitely generated indecomposable
$k\db{G}$-module with vertex $Q$ need be a summand of $V\ind{G}_Q$
for some finitely generated module $V$.  If not then even the
existence of a source for $U$ is uncertain.  If a finitely generated
source exists then we have the following analogue to the well-known
result for finite groups:

\begin{theorem}\label{sources conjugate}
Let $G$ be a virtually pro-$p$ group and let $U$ be a finitely
generated indecomposable $k\db{G}$-module with vertex $Q$ and finitely generated
source.  If $S,T$ are finitely generated $k\db{Q}$-modules that act
as sources of $U$, then $S\iso x(T)$ for some $x\in\norm{G}(Q)$.
\end{theorem}

\begin{proof}
We work within a cofinal system of $N\lhd_O G$ for which $S_{Q\cap
N}, T_{Q\cap N}, S\ind{QN}$ and $T\ind{QN}$ are indecomposable -
this is allowed by \ref{U_N indec} and \ref{cofinal induced
indecomposable}.  For any $N$ in this system we have
$$U\ds S\ind{G}\implies U\res{QN}\ds S\ind{G}\res{QN}\iso \bigoplus_{z\in QN\backslash G/Q} z(S)\res{zQz^{-1}\cap QN}\ind{QN}.$$
Since $U\ds U\res{QN}\ind{G}$ we must have that some indecomposable
summand $X$ of $U\res{QN}$ has vertex conjugate to $Q$.  If
$zQz^{-1}\cap QN$ is properly contained in $zQz^{-1}$ the summands
of $z(S)\res{zQz^{-1}\cap QN}$ have vertex strictly smaller than a
conjugate of $Q$, and so it follows that for some $z\in G$ with
$zQz^{-1}\subseteq QN$ we have $X\ds z(S)\res{zQz^{-1}\cap
QN}\ind{QN}= z(S)\ind{QN}$.  Note also that
$$zQz^{-1}\subseteq QN\implies zQz^{-1}N\subseteq QN\implies zQNz^{-1}=QN$$
so that $z\in \norm{G}(QN)$.

Since $z(S)\ind{QN}\iso z(S\ind{QN})$ by \ref{comm commutes with ind
2} and $S\ind{QN}$ is indecomposable, it follows that $z(S)\ind{QN}$
is indecomposable and so for this $z$ we have
$$z(S)\ind{QN}\ds U\res{QN}.$$

On the other hand $U\ds T\ind{G}$, so
$$z(S)\ind{QN}\ds T\ind{G}\res{QN}\iso \bigoplus_{y\in QN\backslash G/Q} y(T)\res{yQy^{-1}\cap QN}\ind{QN}$$
and by the same argument we find that $z(S)\ind{QN}\iso
y(T)\ind{QN}$ for some element $y\in \norm{G}(QN)$.

Note that $zQz^{-1}/(zQz^{-1}\cap N)\iso zQNz^{-1}/N\iso QN/N$.  We
will use this observation and \ref{comm commutes with coinv} to
transfer these results from the setting of induced modules to the
setting of coinvariant modules where we have the necessary tools to
draw the conclusions we require:

For each $N\lhd_O G$ in our inverse system we have
\begin{align*}
z(S)\ind{QN}                           \iso & y(T)\ind{QN}             \\
\implies (z(S)\ind{QN})_N              \iso & (y(T)\ind{QN})_N \\
\implies z(S)_{zQz^{-1}\cap N} \iso & y(T)_{yQy^{-1}\cap N}       \\
\implies z^{-1}(z(S)_{zQz^{-1}\cap N}) \iso & z^{-1}(y(T)_{yQy^{-1}\cap N})    \\
\implies S_{Q\cap N}             \iso & z^{-1}y(T)_{(z^{-1}y)Q(z^{-1}y)^{-1}\cap N}    \\
\implies S_{Q\cap N}             \iso & z^{-1}y(T_{Q\cap N}).
\end{align*}

Denote by $C_N$ the set of $w\in \norm{G}(QN)$ such that $S_{Q\cap
N}\iso w(T_{Q\cap N})$.  Since $z^{-1}y$ (which depends on $N$)
satisfies these conditions it follows that $C_N$ is non-empty. Each
$C_N$ is also clearly closed in $G$. The theorem follows easily once
we show the intersection $\bigcap_N C_N$ is non-empty.

Certainly $C_{N_1\cap\hdots \cap N_n}\neq\emptyset$ for any finite
set $N_1,\hdots, N_n$.  Let $N_1\cap\hdots \cap N_n=M$ and fix $w\in
C_M$.  Now $M\leq N_i$ for each $i$ and so
\begin{align*}
S_{Q\cap M} &\iso w(T_{Q\cap M})         \\
\implies (S_{Q\cap M})_{QM\cap N_i} &\iso w(T_{Q\cap M})_{QM\cap N_i}  \\
\implies S_{Q\cap N_i}&\iso w(T_{Q\cap N_i})
\end{align*}
by \ref{coinvariant technicals} so that $w\in C_{N_i}$ for each $i$,
and so $w\in C_{N_1}\cap\hdots\cap C_{N_n}$.  Thus, by compactness
we have $\bigcap_N C_N\neq \emptyset$.

Fix $x\in \bigcap_N C_N$, so that for each $N$ in our system we have
$$S_{Q\cap N}\iso x(T_{Q\cap N}),$$
and since $x\in \bigcap_N\norm{G}(QN)=\norm{G}(Q)$, we can rewrite
this isomorphism as
$$S_{Q\cap N}\iso x(T)_{Q\cap N}$$
so that by \ref{fg summands of coinvs are summands} we have $S\iso
x(T)$, as required.
\end{proof}

\section{Green's Indecomposability Theorem}\label{green indecomp section}

Green's indecomposability theorem says that if $V$ is a finitely
generated absolutely indecomposable module for the group algebra
$kH$, where $H$ is a subnormal subgroup of the finite group $F$ of
index a power of $p$, then the module $V\ind{F}$ is also absolutely
indecomposable. We extend this result to modules over the completed
group algebra of a virtually pro-$p$ group $G$.

Throughout this section let $G$ be a virtually pro-$p$ group and let
$U$ be an indecomposable finitely generated profinite
$k\db{G}$-module. By \ref{coinv non vanishing} and \ref{U_N indec}
we can choose a cofinal inverse system of $N\lhd_O G$ with $U_N$
non-zero and indecomposable, and we will work within this system
throughout.  All rings we consider have a 1.  We do not allow 1 to
equal 0.

For each $N$ in our system let $E_N=\End_{k\db{G}}(U_N)$,
$R_N=\rad(\End_{k\db{G}}(U_N))$ and $\tilde{E}_N=E_N/R_N$.  Each
$E_N$ is a local ring and thus $\tilde{E}_N$ is a finite division
ring \cite[5.21]{curtisandreiner1} so is a finite field. It is clear
that this field must contain $k$.  Our aim for the next few lemmas
is to show that $\End_{k\db{G}}(U)/\rad(\End_{k\db{G}}(U))\iso
\invlim \tilde{E}_N$.

Define maps $\rho_{MN}:E_N\to E_M$ whenever $N\leq M$ as follows: If
$\alpha_N\in E_N$ then define $\rho_{MN}(\alpha_N)=\alpha_M\in E_M$
by $\alpha_M(1\ctens_{M}u)=1\ctens_M\alpha_N(u)$.  Each $\rho_{MN}$
is a ring homomorphism.

\begin{lemma}\label{finite radical to radical}
The map $\rho_{MN}$ sends the radical $R_N$ of $E_N$ into $R_M$, and
thus induces a map $\tilde{\rho}_{MN}:\tilde{E}_N\to \tilde{E}_M$,
which is a ring homomorphism.
\end{lemma}

\begin{proof}
This is easily checked by noting that elements of the radical $R_N$
are precisely the nilpotent endomorphisms of $U_N$.
\end{proof}

Observe that $\{E_N,\rho_{MN}\}$ is an inverse system of finite
rings and $\{\tilde{E}_N,\tilde{\rho}_{MN}\}$ is an inverse system
of finite fields.  Since field homomorphisms are injective we can
choose a cofinal inverse system of $N$ for which every
$\tilde{E}_N=k'$, for some fixed finite extension field $k'$ of $k$.
From now on we will work inside this cofinal inverse system.

Define $E=\End_{k\db{G}}(U), R=\rad(E), \tilde{E}=\tilde{E}(U)=E/R$.
Note that using the universal property of $(-)_N$ a simple tweaking
of \ref{Hom(U,W)isinverselimit} shows that $E\iso\invlim_N E_N$.
Denote by $\rho_N$ the map $E\to E_N$ from the above limit.  This is
the map given by applying the functor $(-)_N$ to the morphisms in
$E$.

\begin{lemma}
The radical of $E$ maps into the radical of $E_N$ under $\rho_N$,
for each $N$.
\end{lemma}

\begin{proof}
Recall that the radical of $E$ consists of all non-invertible
endomorphisms of $U$.  Fix an element $\alpha$ in the radical of
$E$, so that $\alpha$ is not an isomorphism.  If $\alpha$ were surjective, then each
$\alpha_N$ would also be onto because $(-)_N$ is right exact.  But then each $\alpha_N$ would be
an isomorphism, and hence so would be $\alpha$, contrary to assumption.  It follows that $\alpha$ is not surjective.
If each $\rho_N(\alpha)=\alpha_N$ were onto then so would be $\alpha$, so we
can find some $N_0\lhd_O G$ with $\alpha_{N_0}$ not onto.

We note that for any $N'\lhd_O G$ contained in $N_0$, the
corresponding $\alpha_{N'}$ is not onto. Fix some arbitrary $N\lhd_O
G$, and consider $L=N\cap N_0$.  Then $\alpha_L$ is not onto since
$L\leq N_0$, so that $\alpha_L\in R_L$.  But now by Lemma
\ref{finite radical to radical} this implies that $\alpha_N\in R_N$
as well, so that the image of $R$ in $E_N$ is contained inside
$R_N$.
\end{proof}

The endomorphism ring of $U$ is local by \ref{fg module has LER},
and thus $\tilde{E}$ is a division ring.

\begin{lemma}\label{End by rad inverse limit}
The division ring $\tilde{E}$ is a finite field and is isomorphic to
$\invlim\tilde{E}_N$.
\end{lemma}

\begin{proof}
For each $N$ we have canonical surjections
$\gamma_N:E_N\twoheadrightarrow E_N/R_N=\tilde{E}_N$, which give a
map of inverse systems since for $N\leq M$ the diagrams
$$
\xymatrix{
E_N \ar[d]^{\rho_{MN}} \ar@{->>}[r]^{\gamma_N} &\tilde{E}_N\ar[d]^{\tilde{\rho}_{MN}}\\
E_M \ar@{->>}[r]^{\gamma_M} &\tilde{E}_M}
$$
commute.  This map of inverse systems gives a surjection of rings
$\gamma$ from $E$ to $\invlim\tilde{E}_N$.

We note now since $\rho_{N}(R)\subseteq R_N$ for each $N$, that
$R\subseteq \ker(\gamma)$.  Hence, we can factor out $R$ to obtain a
surjection from $E/R=\tilde{E}$ to $\invlim\tilde{E}_N$.  But this
is now a surjection of division rings and hence an isomorphism of
fields, as required.
\end{proof}

If $F$ is a finite group, recall that a $kF$-module $W$ is said to
be \emph{absolutely indecomposable} if the $k'F$-module $k'\otimes_k
W$ is indecomposable for all field extensions $k'$ of $k$.  By
\cite[30.29]{curtisandreiner1}, $W$ is absolutely indecomposable if
and only if $\tilde{E}(W)\iso k$.  We thus have the following immediate
corollary to \ref{End by rad inverse limit}:

\begin{corol}\label{coinv abs indec}
If $G$ is a virtually pro-$p$ group and $U$ is a finitely generated
$k\db{G}$-module with corresponding $\tilde{E}\iso k$, then $U$ is
the inverse limit of an inverse system of finite absolutely
indecomposable modules.
\end{corol}

From \cite[7.14, 3.34, 30.27]{curtisandreiner1} we can make several
important deductions.  Firstly if $F$ is a finite group and $W$ is a
finitely generated $kF$-module, then $W$ is absolutely
indecomposable if and only if $k'\otimes_k W$ is indecomposable for
all \emph{finite} field extensions $k'$ of $k$.  Secondly, if $W$ is
not absolutely indecomposable then the extension $l$ of the field
$k$ required for $l\otimes_k W$ to decompose does not depend
directly on $F$ or $W$, but only on the field $\tilde{E}(W)$.  These
facts ensure that the following definition is appropriate:

\begin{defn}\label{profinite abs indec defn}
A finitely generated profinite $k\db{G}$-module $U$ is
\emph{absolutely indecomposable} if the $k'\db{G}$-module
$k'U=k'\ctens_k U$ is indecomposable for all finite field extensions
$k'$ of $k$.
\end{defn}

\begin{theorem}\label{profinite absindec iff End/rad is k}
If $G$ is a virtually pro-$p$ group, then a finitely generated
$k\db{G}$-module $U$ is absolutely indecomposable if and only if
$\tilde{E}\iso k$.
\end{theorem}

\begin{proof}
If $\tilde{E}\iso k$ then by \ref{coinv abs indec}, $U$ is the inverse limit of a cofinal inverse
system of absolutely indecomposable modules $U_N$. Suppose that
$k'\otimes_k U$ decomposes as $X\oplus Y$ for some finite extension
field $k'$ of $k$ and some $X,Y\neq 0$. Then
$$k'\otimes_k U_N\iso (k'\otimes_k U)_N\iso (X\oplus Y)_N\iso X_N\oplus Y_N.$$
But $X_N$ and $Y_N$ are non-zero since $N$ is pro-$p$, by \ref{coinv
non vanishing}, contradicting the absolute indecomposability of
$U_N$.

To show the forward implication, assume that $\tilde{E}=k'$ for $k'$ a finite field extension of
$k$ which properly contains $k$.  Since $\tilde{E}\iso\invlim
\tilde{E}_N$ we have a cofinal inverse system of modules $U_N$ for
which $\tilde{E}_N=k'$.

By the discussion prior to \ref{profinite abs indec defn} there is a
fixed finite extension field $l$ of $k$ for which each $l\ctens_k
U_N$ decomposes.  But
$$\invlim (l\ctens_k U_N)\iso l\ctens_k \invlim U_N=l\ctens_k U$$
since by \cite[5.5.2]{ribzal} complete tensoring commutes with
$\invlim$ and the actions of $l$ and $G$ carry through this
isomorphism. Now the contrapositive of \ref{U_N indec} demonstrates
that $l\ctens U$ decomposes, so that $U$ is not absolutely
indecomposable.
\end{proof}

We can now prove Green's indecomposability theorem for virtually
pro-$p$ groups:

\begin{theorem}\label{profinite indecomposability}
Let $G$ be a virtually pro-$p$ group, let $H\lhd_C G$ with $G/H$ a
pro-$p$ group, and let $V$ be a finitely generated absolutely
indecomposable $k\db{H}$-module.  Then $V\ind{G}$ is absolutely
indecomposable.
\end{theorem}

\begin{proof}
Suppose for contradiction that $V\ind{G}$ decomposes, so that
$V\ind{G}=X\oplus Y$ for $k\db{G}$-modules $X,Y\neq 0$.  By
\ref{profinite absindec iff End/rad is k} the module $V$ has
corresponding $\tilde{E}(V)=k$, so by \ref{coinv abs indec} we can
find some open normal pro-$p$ subgroup $N$ of $G$ with $V_{H\cap N}$
absolutely indecomposable. Then
$$V_{H\cap N}\ind{G/N} \iso (V\ind{G})_N = (X\oplus Y)_N \iso X_N\oplus Y_N$$
where $X_N, Y_N\neq 0$ by \ref{coinv non vanishing}, so that
$V_{H\cap N}\ind{G/N}$ decomposes.  But this decomposition
contradicts Green's indecomposability theorem for finite groups
\cite[19.23]{curtisandreiner1}, and so $V\ind{G}$ must be
indecomposable.

For absolute indecomposability note that there is a cofinal inverse
system of $N\lhd_O G$ for which $\tilde{E}(V_{H\cap N}\ind{G/N})\iso
k$.  But $V_{H\cap N}\ind{G/N}\iso (V\ind{G})_N$ so that
$\tilde{E}((V\ind{G})_N)\iso k$ for each $N$.  Now $V\ind{G}$ is
absolutely indecomposable by \ref{End by rad inverse limit} and
\ref{profinite absindec iff End/rad is k}.
\end{proof}

As for finite groups we have immediate corollaries:

\begin{corol}
Let $G$ be a virtually pro-$p$ group, let $H\leq_C G$ be subnormal
in $G$ with $|G:H|$ a (possibly infinite) power of $p$, and let $V$
be a finitely generated absolutely indecomposable $k\db{H}$-module.
Then $V\ind{G}$ is absolutely indecomposable.
\end{corol}

\begin{corol}
Let $G$ be a pro-$p$ group, let $H\leq_C G$, and let $V$ be a
finitely generated absolutely indecomposable $k\db{H}$-module. Then
$V\ind{G}$ is absolutely indecomposable.
\end{corol}

\begin{proof}
Each $G/N$ is a finite $p$-group so that $HN/N$ is subnormal in
$G/N$ and the result follows as above.
\end{proof}

We include a virtually pro-$p$ version of a variant of Green's
indecomposability theorem (for the finite case see
\cite{inducedsummandsisomorphic}, \cite{broueinducedsummands} or
\cite{puiginducedsummands}).  It seems a pity that this result is
not widely known for finite groups.

\begin{theorem}\label{profinite inducedsummandsisomorphic}
Let $H$ be a closed subgroup of a virtually pro-$p$ group $G$ and
let $V$ be a finitely generated indecomposable $k\db{H}$-module. If
either $H$ is subnormal in $G$ and of index some (possibly infinite)
power of $p$, or $G$ is pro-$p$, then the indecomposable summands of
$V\ind{G}$ are isomorphic.
\end{theorem}

\begin{proof}
If the module $V\ind{G}$ is indecomposable then we are done.
Otherwise write $V\ind{G}=X\oplus Y\oplus Z$ with $X,Y$ non-zero and
indecomposable. We will show that $X\iso Y$.

Choose a cofinal inverse system of open normal pro-$p$ subgroups $N$
of $G$ so that $V$ itself and the indecomposable summands of
$V\ind{G}$ remain indecomposable on taking coinvariants.  Now for
any such $N$ we have
$$V_{H\cap N}\ind{G/N}\iso (V\ind{G})_N\iso X_N\oplus Y_N\oplus Z_N.$$
But $V_{H\cap N}$ is a finitely generated indecomposable module over
the finite group $HN/N$, and under either hypothesis given above we
have $X_N\iso Y_N$ by \cite{inducedsummandsisomorphic}.  It now
follows immediately from \ref{fg summands of coinvs are summands}
that $X\iso Y$ and we are done.
\end{proof}

\section{Acknowledgements}

The author gratefully acknowledges the support of his
supervisor Peter Symonds throughout this project.  Thanks also to the referee for helpful comments.

\bibliographystyle{plain}
\bibliography{bibliography}

\end{document}